\documentclass[11pt, a4paper]{article}
\usepackage{amsmath}
\usepackage{wasysym}
\usepackage{latexsym}
\usepackage{amsfonts}
\usepackage{mathrsfs}
\usepackage{amssymb}
\usepackage{ifsym}
\usepackage{dsfont}
\usepackage[all]{xy}
\usepackage{authblk}

\newtheorem{theorem}{Theorem}[section]

\newtheorem{definition}[theorem]{Definition}
\newtheorem{lemma}[theorem]{Lemma}
\newtheorem{corollary}[theorem]{Corollary}

\newtheorem{question}[theorem]{Question}

\newenvironment{proof}[1][Proof]{\begin{trivlist}
\item[\hskip \labelsep {\bfseries #1}]}{\end{trivlist}}

\newcommand{\qed}{\nobreak \ifvmode \relax \else
      \ifdim\lastskip<1.5em \hskip-\lastskip
      \hskip1.5em plus0em minus0.5em \fi \nobreak
      \vrule height0.75em width0.5em depth0.25em\fi}

\begin{document}
\title{On a slight weakening of Kripke-Platek Set Theory} 
\author[1]{Zachiri McKenzie}
\affil[1]{University of Chester}
\maketitle

\begin{abstract}
The weak set theory $\mathsf{ReR}$ is obtained from Kripke-Platek Set Theory ($\mathsf{KP}$) by replacing the bounded collection scheme with the bounded replacement scheme. We show that $\mathsf{ReR}$ proves $\mathsf{TCo}$, which asserts that every set is contained in a transitive set. This is used to show that the theories obtained by adding the negation of the axiom of infinity to $\mathsf{ReR}$ and $\mathsf{KP}$ have the same consequences. Our proof of $\mathsf{TCo}$ relies on the availability of a fragment of class foundation in $\mathsf{ReR}$. To demonstrate the necessity of this reliance, even in the presence of infinity, we build a model of a significant fragment of $\mathsf{ZF}$ that includes bounded separation and collection, infinity, powerset, regularity and the axiom of choice, in which $\mathsf{TCo}$ fails.
\end{abstract}

\section[Introduction]{Introduction}

On of the strongest theories studied in \cite{mat06} is the set theory $\mathsf{ReR}$ that is obtained from Kripke-Platek Set Theory (without the Axiom of Infinity) by replacing the bounded collection scheme with the bounded replacement scheme. From the work of Zarach \cite{zar96}, we know that there are theorems provable in Kripke-Platek Set Theory that are not provable in $\mathsf{ReR}$. In this paper, we show that $\mathsf{ReR}$ proves that every set is contained in a transitive set, answering a question \cite[Problem 2.107]{mat06} posed by Mathias. This result is used to show that $\mathsf{ReR}$ is capable of defining and proving the bijectivity of the inverse Ackermann interpretation that describes a correspondence between the hereditarily finite sets and the finite von Neumann ordinals. This is used to show that the theory obtained by adding the negation of the Axiom of Infinity to $\mathsf{ReR}$ is the same as the theory obtained by adding the negation of the Axiom of Infinity to Kripke-Platek Set Theory. In particular, the theory $\mathsf{ReR}$ proves that for all sets $x$, the set of finite subsets of $x$ is a set, answering another question \cite[Problem 8.26]{mat06} posed by Mathias.

In section \ref{Sec:ModelInWhichTCoFails}, we build a model of a significant fragment of set theory including bounded collection and separation, the powerset axiom, the axiom of choice, infinity and foundation for sets, in which the axiom of transitive containment fails. In particular, this shows that if the scheme of foundation for $\Pi_1$-classes is weakened to foundation for sets in Kripke-Platek Set Theory with Infinity, then the resulting theory no longer proves that every set is contained in a transitive set.

\section[Background]{Background}

Throughout this paper $\mathcal{L}$ will denote the language of set theory--- first order logic with equality ($=$) and a binary relation symbol $\in$. Let $\Delta_0$ ($=\Sigma_0=\Pi_0$) be the class of $\mathcal{L}$-formulae whose quantifiers are all bounded by the $\in$ relation. The class $\Delta_0^{\mathcal{P}}$, introduced by Takahashi \cite{tak72}, consist of all $\mathcal{L}$-formulae whose quantifiers are all bounded by either $\in$ or $\subseteq$. The L\'{e}vy classes of $\mathcal{L}$-formulae $\Sigma_n$ and $\Pi_n$ are defined inductively from $\Delta_0$: a formula is $\Sigma_{n+1}$ if it is of the form $\exists x \phi$ where $\phi$ is $\Pi_n$, and a formula is $\Pi_{n+1}$ if it is of the form $\forall x \phi$ where $\phi$ is $\Sigma_n$.

Let $\Gamma$ be a collection of $\mathcal{L}$-formulae. We use $\Gamma\textsf{-Separation}$ and $\Gamma\textsf{-Collection}$ to denote the restrictions of the usual axiom schemes of $\textsf{Separation}$ and $\textsf{Collection}$, respectively, to formulae that are in $\Gamma$. We use ($\Gamma$-)$\textsf{Foundation}$ to denote the scheme asserting that every nonempty class that is the extension of an $\mathcal{L}$-formula ($\Gamma$-formula, respectively) (with parameters) contains an $\in$-minimal element. We write $\textsf{Set-Foundation}$ for the single axiom asserting that every set has an $\in$-minimal element. 

\begin{definition}
Let $\phi(x, \vec{z})$ be an $\mathcal{L}$-formula. We write $\exists ! x \phi(x, \vec{z})$ for the $\mathcal{L}$-formula
\[
\exists x \phi(x, \vec{z}) \land \forall u \forall v (\phi(u, \vec{z}) \land \phi(v, \vec{z}) \Rightarrow u =v).
\]
\end{definition}

We write ($\Gamma$-)$\textsf{Replacement}$ for the scheme consisting of the axioms: for all $\mathcal{L}$-formulae ($\Gamma$-formulae, respectively), $\phi(x, y, \vec{z})$,
\[
\forall \vec{z} \forall u ((\forall x \in u) \exists! y \phi(x, y, \vec{z}) \Rightarrow \exists v \forall y (y \in v \iff (\exists x \in u)\phi(x, y, \vec{z}))).
\]
The axiom $\textsf{Infinity}$ asserts that a superset of $\omega$ exists, and $\textsf{Powerset}$ asserts that for all $x$, the set of all subsets of $x$ exists. As usual, we use $\mathsf{AC}$ to denote the Axiom of Choice. 
\begin{itemize}
\item $\mathsf{S}_0$ is $\mathcal{L}$-theory with axioms \textsf{Extensionality} (two sets are equal if and only if they contain the same elements), and \textsf{Empty Set}, \textsf{Pair}, \textsf{Union} and \textsf{Difference} asserting, respectively, that for all $x$ and $y$, the sets $\emptyset$, $\{x, y\}$, $\bigcup x$ and $x \backslash y$ exist.
\item $\mathsf{ReR}_0$ is obtained from $\mathsf{S}_0$ adding $\Delta_0\textsf{-Replacement}$.
\item $\mathsf{ReR}$ is obtained from $\mathsf{ReR}_0$ by adding $\Pi_1\textsf{-Foundation}$.
\item $\mathsf{KP}$ is obtained from $\mathsf{ReR}$ by replacing $\Delta_0\textsf{-Replacement}$ with $\Delta_0\textsf{-Separation}$ and $\Delta_0\textsf{-Collection}$.
\item The theories $\mathsf{KP}^{\neg \infty}$ and $\mathsf{ReR}^{\neg \infty}$ are obtained from $\mathsf{KP}$ and $\mathsf{ReR}$, respectively, by adding $\neg \textsf{Infinity}$.
\item Zermelo Set Theory ($\mathsf{Z}$) is obtained from $\mathsf{S}_0$ by adding $\textsf{Separation}$, $\textsf{Infinity}$ , $\textsf{Powerset}$ and $\textsf{Set-Foundation}$. 
\end{itemize}
Note that $\mathsf{ReR}_0$ proves $\Delta_0\textsf{-Separation}$. Zarach \cite{zar96} shows that there is a model of $\mathsf{ReR}+\mathsf{Replacement}$ that does not satisfy $\mathsf{KP}$.

The inclusion of $\textsf{Pair}$ in $\mathsf{S}_0$ facilitates the coding of ordered pairs by $\langle x, y \rangle= \{ \{x\}, \{x, y\}\}$. This mean that, in theory $\mathsf{S}_0$, there is a $\Delta_0$-formula $\mathsf{OP}(x)$ saying that ``$x$ is an ordered pair", and functions $\mathsf{fst}(z)$ and $\mathsf{snd}(z)$ defined by $\Delta_0$-formulae such that for all $\langle x, y \rangle$, $\mathsf{fst}(\langle x, y \rangle)=x$ and $\mathsf{snd}(\langle x, y \rangle)= y$. In \cite{gan74}, Gandy identifies a refinement of the class of functions whose graphs are specified by $\Delta_0$-properties. A function $F$ is {\bf substitutable} if for all $\Delta_0$-formulae $\phi(u, \vec{z})$, the formula $(\exists u \in F(\vec{x}))\phi(u, \vec{z})$ is equivalent to a $\Delta_0$-formula. The operations $\mathsf{fst}$ and $\mathsf{snd}$ are substitutable, as are the G\"{o}del operations that generate constructible hierarchy \cite{gan74}. When presenting $\mathcal{L}$-formulae, we make use of the fact that many fundamental set-theoretic notions, such as ``$x$ is a natural number (a finite von Neumann ordinal)", ``$x$ is transitive", ``$f$ is a function", $x= \mathsf{dom}(f)$ and $x= \mathsf{rng}(f)$ can be expressed by $\Delta_0$-formulae. We refer the reader to \cite[Table 1 on p.14]{bar75} which provides a table of renderings of some commonly encountered set-theoretic notions as $\Delta_0$-formulae. 

A set $x$ is {\bf transitive} if for all $y \in x$ and for all $z \in y$, $z \in x$. The Axiom of Transitive Containment ($\mathsf{TCo}$) asserts that every set is contained in a transitive set, i.e.
\[
\forall x \exists y(x \subseteq y \land (\forall z \in y)(\forall w \in z)(w \in y)).
\]
We use $y=\mathsf{tcl}(x)$ to denote the formula that says $x \subseteq y$, $y$ is transitive, and for all $z$, if $z$ is transitive with $x \subseteq z$, then $y \subseteq z$. \cite[Proposition 1.20]{mat01} shows that $\mathsf{KP}$ proves $\mathsf{TCo}$. The proof of \cite[Theorem I.6.1]{bar75} (once one confirms that no more than $\Pi_1\textsf{-Foundation}$ is being used) shows that $\mathsf{KP}$ proves $\forall x \exists y(y= \mathsf{tcl}(x))$. In contrast, it has long been known that there are relatively strong subsystems of Zermelo-Fraenkel Set Theory ($\mathsf{ZF}$) that do not prove $\mathsf{TCo}$. The theory $\mathsf{Z}+\mathsf{TCo}$ proves $\textsf{Foundation}$. Therefore, the work of Jensen and Schr\"{o}der \cite{js69}, and Boffa \cite{bof69, bof70} showing that $\mathsf{Foundation}$ is not provable $\mathsf{Z}$ shows that $\mathsf{Z}$ does not prove $\mathsf{TCo}$. Mathias \cite[\S 12]{mat06} improves this result by showing that $\mathsf{Z}+\textsf{Foundation}$ does not prove $\mathsf{TCo}$. Let $\mathsf{ZF}_{\textsf{fin}}$ be the theory obtained from $\mathsf{ZF}$ by replacing $\textsf{Infinity}$ with $\neg \textsf{Infinity}$. H\'{a}jek and Vop\v{e}nca \cite{hv63}, and Hauschild \cite{hau66} show that $\mathsf{ZF}_{\textsf{fin}}$ does not prove $\mathsf{TCo}$. The issue is that $\mathsf{ZF}_{\textsf{fin}}$ does not prove $\textsf{Foundation}$. The recursion that constructs
\[
\mathsf{tcl}(x)= \bigcup \left\{ x, \bigcup x, \bigcup\bigcup x, \ldots \right\}
\]     
can be carried out in $\mathsf{S}_0+\textsf{Infinity}+\Sigma_1\textsf{-Separation}+\Delta_0\textsf{-Replacement}$ showing that, in the presence of $\textsf{Infinity}$, $\mathsf{TCo}$ follows directly from enough separation and bounded replacement or collection.

Let $\mathcal{M}= \langle M, \in^{\mathcal{M}} \rangle$ be an $\mathcal{L}$-structure, and let $a \in M$. As long as $\mathcal{M}$ is clear from the context, we will use $a^*$ to denote $\{x \in M \mid \mathcal{M} \models (x \in a)\}$, i.e. the extension of the point $a$ in $\mathcal{M}$. Let $\mathcal{N}= \langle N, \in^{\mathcal{N}} \rangle$ be an $\mathcal{L}$-structure with $\mathcal{M}$ a substructure of $\mathcal{N}$.
\begin{itemize}
\item We say that $\mathcal{M}$ is a {\bf transitive substructure} of $\mathcal{N}$, and write $\mathcal{M} \subseteq_{\mathsf{tr}} \mathcal{N}$, if for all $x \in M$ and for all $y \in N$, if $\mathcal{N} \models (y \in x)$, then $y \in M$. 
\item We say that $\mathcal{M}$ is a {\bf supertransitive substructure} of $\mathcal{N}$, and write $\mathcal{M} \subseteq_{\mathsf{tr}}^{\mathcal{P}} \mathcal{N}$, if $\mathcal{M} \subseteq_{\mathsf{tr}} \mathcal{N}$ and for all $x \in M$ and for all $y \in N$, if $\mathcal{N} \models (y \subseteq x)$, then $y \in M$.   
\end{itemize}
If $\mathcal{M} \subseteq_{\mathsf{tr}} \mathcal{N}$, then for all  $\Delta_0$-formulae, $\phi(\vec{x})$, and for all $\vec{a} \in M$, $\mathcal{M} \models \phi(\vec{a})$ if and only if $\mathcal{N} \models \phi(\vec{a})$. Similarly, if $\mathcal{M} \subseteq_{\mathsf{tr}}^{\mathcal{P}} \mathcal{N}$, then for all  $\Delta_0^{\mathcal{P}}$-formulae, $\phi(\vec{x})$, and for all $\vec{a} \in M$, $\mathcal{M} \models \phi(\vec{a})$ if and only if $\mathcal{N} \models \phi(\vec{a})$.  

\section[TCo in ReR]{Transitive closures and finite sets in $\mathsf{ReR}$} \label{Sec:TCoAndFiniteSets}

In this section we will show that the theory $\mathsf{ReR}$ proves that every set has transitive closure. This result is used to show that the theory $\mathsf{ReR}^{\neg \infty}$ is able to define and prove the bijectivity of the Ackermann correspondence between natural numbers and hereditarily finite sets. The availability of this class definable bijection means that $\mathsf{ReR}^{\neg \infty}$ and $\mathsf{KP}^{\neg \infty}$ have the same consequences. 

We begin by confirming that $\mathsf{ReR}$ proves that Cartesian products exists. The proofs of Lemmas \ref{Th:CartesianProductOfSingleton} and \ref{Th:CartesianProduct} use essentially the say argument as the one used to prove \cite[Proposition I.3.2]{bar75}. 

\begin{lemma} \label{Th:CartesianProductOfSingleton}
($\mathsf{ReR}$) For all $x$ and for all $y$, $\{x\} \times y$ exists.  
\end{lemma}

\begin{proof}
Let $x$ and $y$ be sets. Note that $(\forall u \in y) \exists! v(v= \langle x, u \rangle)$. So, using $\Delta_0\textsf{-Replacement}$, there exists $c$ such that for all $v$, $v \in c$ if and only if $(\exists u \in y)(v= \langle x, u\rangle)$.  So, $c= \{x\} \times y$. \hfill$\square$
\end{proof}

\begin{lemma} \label{Th:CartesianProduct}
($\mathsf{ReR}$) For all $x$ and for all $y$, $x \times y$ exists.
\end{lemma}

\begin{proof}
Let $x$ and $y$ be sets. Let $\phi(u, v, y)$ be the $\Delta_0$-formula
\[
(\forall z \in v)(\exists w \in y)(z=\langle u, w\rangle) \land (\forall w \in y)(\exists z \in v)(z= \langle u, w \rangle). 
\]
Then, by Lemma \ref{Th:CartesianProductOfSingleton}, $(\forall u \in x) \exists! v \phi(u, v, y)$. So, using $\Delta_0\textsf{-Replacement}$, let $c$ be such that for all $v$, $v \in c$ if and only if $(\exists u \in x)\phi(u, v, y)$. Then $x \times y= \bigcup c$. \hfill$\square$  
\end{proof} 

Gitman, Hamkins and Johnstone \cite{ghj16} show that there models of full replacement in which the collection of formulae equivalent to a $\Sigma_1$-formula and the collection of formulae equivalent to a $\Pi_1$-formula are not closed under bounded quantification. In order to recover a normalisation procedure that can be carried out in $\mathsf{ReR}$ we identify a more restrictive class of formulae that we will use in place of $\Sigma_1$-formulae.

\begin{definition}
Let $T$ be an $\mathcal{L}$-theory. We say that an $\mathcal{L}$-formula is $\Sigma_{\mathsf{u}}$ if it is $\exists ! x \phi(x, \vec{z})$ where $\phi(x, \vec{z})$ is $\Delta_0$. We say that $\phi(\vec{z})$ is $\Delta_{\mathsf{u}}^T$ if $T$ proves that both $\phi(\vec{z})$ and $\neg \phi(\vec{z})$ are equivalent to $\Sigma_{\mathsf{u}}$-formulae. 
\end{definition}

\begin{lemma} \label{Th:UniqueSigmaNormalisation}
($\mathsf{ReR}$)
\begin{itemize}
\item[(I)] If $\phi(x, y, \vec{z})$ is $\Delta_0$, then $\exists! x \exists! y \phi(x, y, \vec{z})$ is equivalent to a $\Sigma_{\mathsf{u}}$-formula.
\item[(II)] If $\phi(x, y, \vec{z})$ is $\Delta_0$, then $(\forall x \in u) \exists! y \phi(x, y, \vec{z})$ is equivalent to a $\Sigma_{\mathsf{u}}$-formula.
\end{itemize}
\end{lemma}

\begin{proof}
Let $\phi(x, y, \vec{z})$ be a $\Delta_0$-formula. To see that (I) holds, let $\psi(w, \vec{z})$ be the $\Delta_0$-formula
\[
\mathsf{OP}(w) \land \phi(\mathsf{fst}(w), \mathsf{snd}(w), \vec{z}).
\]
Then for all $\vec{z}$, 
\[
\exists! w \psi(w, \vec{z}) \textrm{ if and only if } \exists! x \exists! y \phi(x, y, \vec{z}).
\]
To see that (II) holds, let $\theta(f, u, \vec{z})$ be the $\Delta_0$-formula
\[
(f \textrm{ is a function}) \land \mathsf{dom}(f)=u \land (\forall  x \in u) \phi(x, f(x), \vec{z}).
\]
Now, let $u, \vec{z}$ be sets. If $\exists! f \theta(f, u, \vec{z})$, then $(\forall x \in u) \exists! y \phi(x, y, \vec{z})$ holds. Conversely, assume that $(\forall x \in u) \exists! y \phi(x, y, \vec{z})$ holds. Using $\Delta_0\textsf{-Replacement}$ and Lemma \ref{Th:CartesianProduct}, we can find $f$ such that $\theta(f, u, \vec{z})$. Since $(\forall x \in u) \exists! y \phi(x, y, \vec{z})$, this $f$ is uniquely determined. \hfill$\square$  
\end{proof}

\begin{lemma} \label{Th:ClosurePropertiesDeltaU}
Let $T$ be an $\mathcal{L}$-theory such that $T \vdash \mathsf{ReR}$. 
\begin{itemize}
\item[(I)] If $\phi(\vec{x})$ is $\Delta_{\mathsf{u}}^T$, then $\neg \phi(\vec{x})$ is $\Delta_{\mathsf{u}}^T$.
\item[(II)] If $\phi_0(\vec{x})$ and $\phi_1(\vec{y})$ are $\Delta_{\mathsf{u}}^T$, then $\phi_0(\vec{x}) \land \phi_1(\vec{y})$  is $\Delta_{\mathsf{u}}^T$.
\item[(III)] If $\phi(x, \vec{z})$ is $\Delta_{\mathsf{u}}^T$, then $(\forall x \in u)\phi(x, \vec{z})$ is $\Delta_{\mathsf{u}}^T$.  
\end{itemize}
\end{lemma}

\begin{proof}
Note that (I) follows immediately from the definition of $\Delta_{\mathsf{u}}^T$. 

To prove (II), let $\phi_0(\vec{x})$ and $\phi_1(\vec{y})$ be $\Delta_{\mathsf{u}}$-formulae. Let $\theta_0(u, \vec{x})$ and $\theta_1(v, \vec{y})$ be $\Delta_0$ such that
$T$ proves that $\phi_0(\vec{x})$ is equivalent to $\exists! u \theta_0(u, \vec{x})$ and $\phi_1(\vec{y})$ is equivalent to $\exists! v \theta_1(v, \vec{y})$. Therefore, $T$ proves that $\phi_0(\vec{x})\land \phi_1(\vec{y})$ is equivalent to $\exists!u \exists! v(\theta_0(u, \vec{x}) \land \theta_1(v, \vec{y}))$, which, by Lemma \ref{Th:UniqueSigmaNormalisation}, is equivalent to a $\Sigma_{\mathsf{u}}$-formula. Now, let $\psi_0(w, \vec{x})$ and $\psi_1(z, \vec{y})$ be $\Delta_0$-formulae such that $T$ proves that $\neg \phi_0(\vec{x})$ is equivalent to $\exists! w \psi_0(w, \vec{x})$ and $\neg \phi_1(\vec{y})$ is equivalent to $\exists! z \psi_1(z, \vec{y})$. The theory $T$ proves that $\neg(\phi_0(\vec{x}) \land \phi_1(\vec{y}))$ is equivalent to
\[
\exists! u \exists! v ((\psi_0(u, \vec{x}) \land \psi_1(v, \vec{y})) \lor (\psi_0(u, \vec{x}) \land \theta_1(v, \vec{y})) \lor (\theta_0(u, \vec{x}) \land \psi_1(v, \vec{y}))),
\]
which, by Lemma \ref{Th:UniqueSigmaNormalisation} is equivalent to a $\Sigma_{\mathsf{u}}$-formula.

To prove (III), let $\phi(x, \vec{z})$ be a $\Delta_{\mathsf{u}}^T$-formula. Let $\theta(v, x, \vec{z})$ be $\Delta_0$ such that $T$ proves that $\phi(x, \vec{z})$ is equivalent to $\exists! v \theta(v, x, \vec{z})$. Therefore, $T$ proves that $(\forall x \in u)\phi(x, \vec{z})$ is equivalent to $(\forall x \in u) \exists! v \theta(v, x, \vec{z})$, which, by Lemma \ref{Th:UniqueSigmaNormalisation}, is equivalent to a $\Sigma_{\mathsf{u}}$-formula. Let $\psi(w, x, \vec{z})$ be $\Delta_0$ such that $T$ proves that $\neg \phi(x, \vec{z})$ is equivalent to $\exists! w \psi(w, x, \vec{z})$. So, $T$ proves that $\neg (\forall x \in u) \phi(x, \vec{z})$ is equivalent to $\exists! w (\exists x \in u) \psi(w, x, \vec{z})$, which is a $\Sigma_{\mathsf{u}}$-formula. \hfill$\square$        
\end{proof}

Equipped with the notions of $\Sigma_{\mathsf{u}}$- and $\Delta_{\mathsf{u}}$-formulae, we are able to prove strengthenings of the schemes of $\Delta_0\textsf{-Separation}$ and $\Delta_0\textsf{-Replacement}$ in $\mathsf{ReR}$.

\begin{theorem}
($\mathsf{ReR}$) $\Sigma_{\mathsf{u}}\textsf{-Replacement}$.
\end{theorem}

\begin{proof}
Work in the theory $\mathsf{ReR}$. Let $\phi(w, x, y, \vec{z})$ be a $\Delta_0$-formula. Let $b$ and $\vec{a}$ be such that $(\forall x \in b)\exists! y \exists! w \phi(w, x, y, \vec{a})$. Let $\theta(v, x, \vec{z})$ be the $\Delta_0$-formula
\[
(v \textrm{ is an ordered pair}) \land \phi(\mathsf{snd}(v), x, \mathsf{fst}(v), \vec{z}).
\]
So, $(\forall x \in b) \exists! v \theta(v, x, \vec{a})$. Using $\Delta_0\textsf{-Replacement}$, let $c$ be such that for all $v$,
\[
v \in c \textrm{ if and only if } (\exists x \in b)\theta(v, x, \vec{a}).
\]
Let $d= \mathsf{dom}(c)$. Then for all $y$,
\[
y \in d \textrm{ if and only if } (\exists x \in b) \exists! w \phi(w, x, y, \vec{a}).
\] 
This shows that $\Sigma_{\mathsf{u}}\textsf{-Replacement}$ holds. \hfill$\square$
\end{proof}

The following scheme is a strengthening of $\Delta_0\textsf{-Separation}$.
\begin{itemize}
\item[]($\Delta_{\mathsf{u}}\textsf{-Separation}$) For all $\Delta_0$-formulae $\phi(y, x, \vec{z})$ and $\theta(y, x, \vec{z})$,
\[
\begin{array}{c}
\forall \vec{z} \forall x(\exists! y \phi(y, x, \vec{z}) \iff \exists! y \theta(y, x, \vec{z})) \Rightarrow\\
\forall \vec{z} \forall w \exists v \forall x (x \in v \iff x \in w \land \exists y \phi(y, x, \vec{z}))
\end{array}.
\]
\end{itemize}

\begin{theorem}
($\mathsf{ReR}$) $\Delta_{\mathsf{u}}\textsf{-Separation}$.
\end{theorem}

\begin{proof}
Work in the theory $\mathsf{ReR}$. Let $\phi(y, x, \vec{z})$ and $\theta(y, x, \vec{z})$ be $\Delta_0$-formulae such that for all $x$ and for all $\vec{z}$, $\forall u \forall v(\phi(u, x, \vec{z}) \land \phi(v, x, \vec{z}) \Rightarrow u=v)$, $\forall u \forall v(\theta(u, x, \vec{z}) \land \theta(v, x, \vec{z}) \Rightarrow u=v)$ and $\exists y \phi(y, x, \vec{z}) \iff \neg \exists y \theta(y, x, \vec{z})$. Let $w$ and $\vec{a}$ be a sets. So, $(\forall x \in w) \exists! y(\phi(y, x, \vec{a}) \lor \theta(y, x, \vec{a}))$. So, using $\Delta_0\textsf{-Replacement}$, let $c$ be a set such that for all $y$, 
\[
y \in c \textrm{ if and only if } (\exists x \in w)(\phi(y, x, \vec{a}) \lor \theta(y, x, \vec{a})).
\]
Now, let $d= \{ x \in w \mid (\exists y \in c) \phi(y, x, \vec{a})\}$, which is a set by $\Delta_0\textsf{-Separation}$. Note that for all $x \in w$, $x \in d$ if and only if $\exists y \phi(y, x, \vec{a})$. This shows that $\Delta_{\mathsf{u}}\textsf{-Separation}$ holds. \hfill$\square$    
\end{proof} 

Utilising $\Pi_1\textsf{-Foundation}$, we can show that every finite approximation of the transitive closure exists.

\begin{lemma} \label{Th:FiniteApproximationsOfTCl}
($\mathsf{ReR}$) For all $x$ and for all natural numbers $n$, there exists a unique function $f$ with domain $n+1$ such that $f(0)=x$ and for all $k \in n$, 
\[
f(k+1)= f(k) \cup \bigcup f(k).
\]
\end{lemma}

\begin{proof}
Let $x$ be a set. Let $\phi(n, f, x)$ be the $\Delta_0$-formula 
\[
(n \textrm{ is a natural number}) \Rightarrow \left(\begin{array}{c}
(f \textrm{ is a function}) \land \mathsf{dom}(f)=n+1 \land\\
f(0)= x \land (\forall k \in n)\left(f(k+1)= f(k) \cup \bigcup f(k)\right) 
\end{array} \right).
\]
A straightforward induction shows that if $n$ is a natural number and $f$ and $f^\prime$ are such that $\phi(n, f, x)$ and $\phi(n, f^\prime, x)$, then $f=f^\prime$. Therefore, if the Lemma were false, then $\Pi_1\textsf{-Foundation}$ would yield a least natural number $k+1$ such that $\neg \exists f \phi(k+1, f, x)$, which is impossible. \hfill$\square$
\end{proof}

\begin{lemma} \label{Th:TClMembershipComplexity}
($\mathsf{ReR}$) The formula $u \in \mathsf{tcl}(x)$ is expressible by a $\Sigma_{\mathsf{u}}$-formula.
\end{lemma}

\begin{proof}
Let $\phi(f, u, x)$ be the $\Delta_0$-formula
\[
\begin{array}{c}
(f \textrm{ is a function}) \land \mathsf{dom}(f)=n+1 \land (f(0)=x) \land (u \in f(n)) \land\\ 
(n=0 \lor u \notin f(n-1)) \land (\forall k \in n) \left(f(k+1)= f(k) \cup  \bigcup f(k) \right)
\end{array}.
\]
Now, for all $x$ and $u$, Lemma \ref{Th:FiniteApproximationsOfTCl} and \textsf{Set-Foundation} ensure that $u \in \mathsf{tcl}(x)$ if and only $\exists! f \phi(f, u, x)$. \hfill$\square$  
\end{proof}

\begin{lemma} \label{Th:TClSigmaComplexity}
($\mathsf{ReR}$) The formula is $y=\mathsf{tcl}(x)$ is expressible by a $\Sigma_{\mathsf{u}}$-formula. 
\end{lemma}

\begin{proof}
We have $y = \mathsf{tcl}(x)$ if and only if 
\[
(y \textrm{ is transitive}) \land (x \subseteq y) \land (\forall u \in y)(u \in \mathsf{tcl}(x)).
\]
Using Lemmas \ref{Th:UniqueSigmaNormalisation} and \ref{Th:TClMembershipComplexity}, this is equivalent to a $\Sigma_{\mathsf{u}}$-formula. \hfill$\square$
\end{proof}

This allows us to give a positive answer to \cite[Problem 2.107]{mat06}.

\begin{theorem} \label{Th:TCoInReR}
($\mathsf{ReR}$) For all $x$, $\mathsf{tcl}(x)$ exists.
\end{theorem}

\begin{proof}
Suppose that the theorem is false. Using Lemma \ref{Th:TClSigmaComplexity}, let $\phi(u, x, y)$ be a $\Delta_0$-formula such that $\exists! u \phi(u, x, y)$ expresses $y= \mathsf{tcl}(x)$. Using $\Pi_1\textsf{-Foundation}$, let $b$ be an $\in$-minimal element of the class $\{ x \mid \neg \exists y(y= \mathsf{tcl}(x))\}$. So,
\[
(\forall x \in b)\exists! w (\mathsf{OP}(w) \land \phi(\mathsf{fst}(w), x, \mathsf{snd}(w))).
\]
Therefore, using $\Delta_0\textsf{-Replacement}$, let $c$ be such that 
\[
\forall w( w \in c \iff \mathsf{OP}(w) \land \phi(\mathsf{fst}(w), x, \mathsf{snd}(w))).
\]
Now, $\mathsf{tcl}(b)= b \cup \bigcup \mathsf{rng}(c)$, which is a contradiction. \hfill$\square$ 
\end{proof}

\begin{corollary} \label{Th:ComplexityOfTCl}
The formula $y= \mathsf{tcl}(x)$ is $\Delta_{\mathsf{u}}^{\mathsf{ReR}}$.
\end{corollary}

\begin{proof}
Lemma \ref{Th:TClSigmaComplexity} shows that $y= \mathsf{tcl}(x)$ is $\Sigma_{\mathsf{u}}$. To see that $y \neq \mathsf{tcl}(x)$ is $\Sigma_{\mathsf{u}}$, let $\phi(u, x, y)$ be a $\Delta_0$-formula such that $\exists u! \phi(u, x, y)$ is equivalent to $y= \mathsf{tcl}(x)$ in $\mathsf{ReR}$. Now, $y \neq \mathsf{tcl}(x)$ if and only if
\[
\exists! w(\mathsf{OP}(w) \land \phi(\mathsf{fst}(w), x, \mathsf{snd}(w)) \land \mathsf{snd}(w) \neq y).
\] 
\hfill$\square$
\end{proof}

Using Theorem \ref{Th:TCoInReR} we can show that the theory $\mathsf{ReR}^{\neg \infty}$ is able to define the inverse Ackermann interpretation (see \cite[\S 6]{kw07}) that describes a bijection between the sets and the class of natural numbers. A consequence of this is that $\mathsf{ReR}^{\neg \infty}$ and $\mathsf{KP}^{\neg \infty}$ have the same consequences. 

As noted in \cite[Theorem 5]{kw07}, the negation of the axiom of infinity immediately implies that every ordinal is a successor ordinal.

\begin{lemma} \label{Th:EveryOrdinalASuccessor}
($\mathsf{ReR}^{\neg \infty}$) For all ordinals $\alpha \neq 0$, there exists $\beta \in \alpha$ such that $\alpha= \beta \cup \{\beta\}$. \hfill$\square$
\end{lemma}

A consequence of this result is that any set of ordinals has a maximal element \cite[Corollary 6]{kw07}.  

\begin{lemma} \label{Th:EverySetOfOrdinalsHasMaximum}
($\mathsf{ReR}^{\neg \infty}$) If $x$ is a nonempty set of ordinals, then $\bigcup x \in x$. \hfill$\square$
\end{lemma}

In order to discuss the inverse Ackermann interpretation we need to recall the interpretation of arithmetic in the finite ordinals in set theory. Let $\Psi_+(f, x, y, z)$ be the $\Delta_0$-formula
\[
\begin{array}{c}
(x, y \textrm{ and } z \textrm{ are natural numbers}) \land (f \textrm{ is a bijective function}) \land\\ 
(\mathsf{dom}(f)= x \times \{0\} \cup y \times \{1\}) \land (\mathsf{rng}(f)=z) \land\\
(\forall u, v \in \mathsf{dom}(f))\left( \begin{array}{c}
f(u) \in f(v) \iff\\
(\mathsf{snd}(u) \in \mathsf{snd}(v)) \lor (\mathsf{snd}(u)=\mathsf{snd}(v) \land \mathsf{fst}(u) \in \mathsf{fst}(v)) 
\end{array}\right)
\end{array}.
\]
Let $\Psi_{\cdot}(f, x, y, z)$ be the $\Delta_0$-formula
\[
\begin{array}{c}
(x, y \textrm{ and } z \textrm{ are natural numbers}) \land (f \textrm{ is a bijective function}) \land\\ 
(\mathsf{dom}(f)= x \times y) \land (\mathsf{rng}(f)= z) \land\\
(\forall u, v \in \mathsf{dom}(f))\left( \begin{array}{c}
f(u) \in f(v) \iff\\
(\mathsf{snd}(u) \in \mathsf{snd}(v)) \lor (\mathsf{snd}(u)=\mathsf{snd}(v) \land \mathsf{fst}(u) \in \mathsf{fst}(v)) 
\end{array}\right)
\end{array}.
\]
For all natural number $x$, $y$ and $z$, define
\[
x+y=z \textrm{ if and only if } \exists f \Psi_+(f, x, y, z) \textrm{ and}
\]
\[
x\cdot y=z \textrm{ if and only if } \exists f \Psi_{\cdot}(f, x, y, z).
\]

A straightforward induction argument yields:

\begin{lemma} \label{Th:UniquenessOfArithmeticAttempts}
($\mathsf{ReR}$) Let $x$, $y$ and $z$ be natural numbers. 
\begin{itemize}
\item[(I)] If $f$, $g$, $z_0$ and $z_1$ are such that $\Psi_+(f, x, y, z_0)$ and $\Psi_+(g, x, y, z_1)$, then $f=g$ and $z_0=z_1$.
\item[(II)] If $f$, $g$, $z_0$ and $z_1$ are such that $\Psi_{\cdot}(f, x, y, z_0)$ and $\Psi_{\cdot}(g, x, y, z_1)$, then $f=g$ and $z_0=z_1$. 
\end{itemize}
\hfill$\square$
\end{lemma}

\begin{lemma} \label{Th:UniquenessOfArithmeticOperations}
($\mathsf{ReR}$) For all natural numbers $k$ and $m$,
\begin{itemize}
\item[(I)] $\exists! n (n= k+m)$;
\item[(II)] $\exists! n(n=k\cdot m)$.
\end{itemize}
\hfill$\square$
\end{lemma}

This shows that addition and multiplication are $\Delta_{\mathsf{u}}^{\mathsf{ReR}}$.

\begin{lemma} \label{Th:ComplexityOfArithmetic}
The formulae $z= x+y$ and $z=x\cdot y$ are $\Delta_{\mathsf{u}}^{\mathsf{ReR}}$.      
\end{lemma} 

\begin{proof}
The fact that the formulae $z=x+y$ and $z= x \cdot y$ are $\Sigma_{\mathsf{u}}$ follows immediately from Lemma \ref{Th:UniquenessOfArithmeticAttempts}. Now, for all $x$, $y$, $z$, $\neg(z= x+y)$ if and only if
\[
\exists! f \exists! w ((\neg (x, y \textrm{ and } z \textrm{ are natural numbers}) \land f=w=\emptyset)\lor (\Psi_+(f, x, y, w) \land w \neq z)),
\]
which, by Lemma \ref{Th:UniqueSigmaNormalisation}, is equivalent to a $\Sigma_{\mathsf{u}}$-formula. Similarly, for all $x$, $y$ and $z$, $\neg(z=x\cdot y)$ if and only if
\[
\exists! f \exists! w ((\neg (x, y \textrm{ and } z \textrm{ are natural numbers}) \land f=w=\emptyset)\lor (\Psi_\cdot(f, x, y, w) \land w \neq z)),
\]
which, by Lemma \ref{Th:UniqueSigmaNormalisation}, is equivalent to a $\Sigma_{\mathsf{u}}$-formula. \hfill$\square$
\end{proof}

Equipped with addition and multiplication, we can now define exponentiation. Let $\Psi_{\mathsf{exp}}(f, x, y)$ be the formula
\[
\begin{array}{c}
(x \textrm{ and } y \textrm{ are natural numbers}) \land (f \textrm{ is a function}) \land \mathsf{dom}(f)=x+1 \land\\
f(0)=1 \land (\forall k \in x)(f(k+1)= 2\cdot f(k)) \land f(x)=y
\end{array}. 
\]
For all natural numbers $x$ and $y$, define
\[
y= 2^x \textrm{ if and only if } \exists f \Psi_{\mathsf{exp}}(f, x, y).
\]
Lemma \ref{Th:ClosurePropertiesDeltaU} and similar arguments to those used to obtain Lemmas \ref{Th:UniquenessOfArithmeticAttempts} and \ref{Th:ComplexityOfArithmetic} yields:

\begin{lemma} \label{Th:ComplexityOfPowers}
The formula $y= 2^x$ is $\Delta_{\mathsf{u}}^{\mathsf{ReR}}$. \hfill$\square$ 
\end{lemma}

\begin{lemma} \label{Th:UniquenessOfPowers}
($\mathsf{ReR}$) For all natural numbers $x$, there exists a unique $y$ such that $y= 2^x$. \hfill$\square$
\end{lemma}

\begin{lemma} \label{Th:PowerFunctionInjective}
($\mathsf{ReR}$) For all natural numbers $x$, $y$ and $z$, if $z= 2^x$ and $z=2^y$, then $x=y$. \hfill$\square$
\end{lemma}

Another important ingredient in the inverse Ackermann interpretation is the function that sums arbitrary sets of natural numbers. Let $\Psi_{\Sigma}(f, x, y, z)$ be the formula
\[
\begin{array}{c}
(x \textrm{ and } z \textrm{ are natural numbers}) \land (y \textrm{ is a set of natural numbers}) \land\\
(f \textrm{ is a function}) \land \mathsf{dom}(f)= x+1 \land f(0)= 0 \land f(x)= z \land\\ 
(\forall k \in x)\left( \begin{array}{c}
(k+1 \in y \Rightarrow f(k+1)=f(k)+k+1) \land\\
(k+1 \notin y \Rightarrow f(k+1)=f(k))
\end{array}\right)
\end{array}.
\]
For all natural numbers $z$ and for all sets of natural numbers $y$, define
\[
z= \sum_{x \in y} x \textrm{ if and only if } \exists f \Psi_{\Sigma}\left(f, \bigcup y, y, z\right). 
\]
Again, with Lemmas \ref{Th:ClosurePropertiesDeltaU} and \ref{Th:EverySetOfOrdinalsHasMaximum}, and the availability of $\textsf{Set-Foundation}$ in $\mathsf{ReR}$ yields:

\begin{lemma} \label{Th:ComplexityOfSum} 
The formula $z= \sum_{x \in y} x$ is $\Delta_{\mathsf{u}}^{\mathsf{ReR}^{\neg \infty}}$. \hfill$\square$
\end{lemma}

\begin{lemma} \label{Th:UniquenessOfSum}
($\mathsf{ReR}^{\neg \infty}$) For all sets of natural numbers $y$, there exists a unique $z$ such that $z= \sum_{x\in y} x$. \hfill$\square$
\end{lemma}

Note that in Lemma \ref{Th:UniquenessOfSum} we are using the fact that, in $\mathsf{ReR}^{\neg \infty}$, every set of natural numbers $x$, $\bigcup x$ is a natural number (Lemma \ref{Th:EverySetOfOrdinalsHasMaximum}).

For all sets $y$, and for functions $g$ with $y \subseteq \mathsf{dom}(g)$ and $\mathsf{rng}(g)$ a set of natural number, define
\[
\begin{array}{lcl}
z= \sum_{v \in y} 2^{g(v)} &\textrm{if and only if}& (g \textrm{ is a function}) \land y \subseteq \mathsf{dom}(g) \land\\ 
&& (\mathsf{rng}(g \upharpoonright y) \textrm{ is a set of natural numbers}) \land\\
&&\exists a\left( \begin{array}{c}
(\forall u \in a)(\exists w \in y)(u= 2^{g(w)}) \land\\
(\forall w \in y)(\exists u \in a)(u=2^{g(w)}) \land z= \sum_{x \in a} x 
\end{array} \right).
\end{array}
\]
In the special case where $y$ is a set of natural numbers and $g$ is the identity function, we write $z= \sum_{v \in y} 2^v$. Note that the theory $\mathsf{ReR}^{\neg \infty}$ ensures that if $g$ is a function with $y \subseteq \mathsf{dom}(g)$ and $\mathsf{rng}(g \upharpoonright y)$ a set of natural numbers, then the set $a$ that satisfies the conjunction
\[
(\forall u \in a)(\exists w \in y)(u= 2^{g(w)}) \land (\forall w \in y)(\exists u \in a)(u=2^{g(w)})
\] 
exists and is unique. Therefore, combined with Lemmas \ref{Th:ClosurePropertiesDeltaU}, \ref{Th:ComplexityOfPowers} and \ref{Th:ComplexityOfSum} we get:

\begin{lemma} \label{Th:ComplexityOfPowerSum}
The formula $z= \sum_{v \in y} 2^{g(v)}$ is $\Delta_{\mathsf{u}}^{\mathsf{ReR}^{\neg \infty}}$. \hfill$\square$
\end{lemma}

\begin{lemma} \label{Th:UniquenessOfPowerSum}
($\mathsf{ReR}^{\neg \infty}$) For all sets $y$ and for all function $g$ with $y \subseteq \mathsf{dom}(g)$ and $\mathsf{rng}(g \upharpoonright y)$ is a set of natural numbers, there exists a unique $z$ such that $z= \sum_{v \in y} 2^{g(v)}$. \hfill$\square$  
\end{lemma}

Formalising the proof that every natural number has a unique representation in binary yields:

\begin{lemma} \label{Th:BinaryRepresentation}
($\mathsf{ReR}^{\neg \infty}$) For all natural numbers $z$, there exists a unique set of natural numbers $y$ such that 
\[
z= \sum_{v \in y} 2^v.
\]
\hfill$\square$
\end{lemma}

We now have everything we need to define the inverse Ackermann operation. Let $\Psi_{\mathsf{IA}}(f, x, y, z)$ be the formula
\[
\begin{array}{c}
y= \mathsf{tcl}(\{x\}) \land (f \textrm{ is a function}) \land \mathsf{dom}(f)= y \land\\
(\forall u \in y)\left(f(u) = \sum_{v \in u} 2^{f(v)} \right) \land f(x)= z
\end{array}.
\]
For all sets $x$ and for all natural numbers $z$, define
\[
z= \mathcal{I}_{\mathsf{Ack}}(x) \textrm{ if and only if } \exists f \exists y \Psi_{\mathsf{IA}}(f, x, y, z).
\]

\begin{lemma} \label{Th:UniquenessOfInverseAckermann}
($\mathsf{ReR}^{\neg \infty}$) Let $x$ be a set. If $f_0$, $y_0$, $f_1$, $y_1$, $z_0$ and $z_1$ are such that $\Psi_{\mathsf{IA}}(f_0, x, y_0, z_0)$ and $\Psi_{\mathsf{IA}}(f_1, x, y_1, z_1)$, then $y_0=y_1$, $f_0=f_1$ and $z_0=z_1$.
\end{lemma}

\begin{proof}
Let $f_0$, $y_0$, $f_1$, $y_1$, $z_0$ and $z_1$ be such that $\Psi_{\mathsf{IA}}(f_0, x, y_0, z_0)$ and $\Psi_{\mathsf{IA}}(f_1, x, y_1, z_1)$. We immediately have $y_0= \mathsf{tcl}(\{x\})= y_1$. Now, let $w$ be an $\in$-minimal element of $\{v \in y_0 \mid f_0(v) \neq f_1(v)\}$. So, for all $u \in w$, $f_0(u)=f_1(u)$. Therefore, by Lemma \ref{Th:UniquenessOfPowerSum},
\[
f_0(w)= \sum_{u \in w} 2^{f_0(u)}= \sum_{u \in w} 2^{f_1(u)}= f_1(w),
\]
which is a contradiction. Therefore $f_0=f_1$ and $z_0=z_1$. \hfill$\square$ 
\end{proof}

\begin{corollary}
($\mathsf{ReR}^{\neg \infty}$) The formula $z= \mathcal{I}_{\mathsf{Ack}}(x)$ is $\Sigma_{\mathsf{u}}$. \hfill$\square$
\end{corollary}

The function $\mathcal{I}_{\mathsf{Ack}}$ defines a bijection between the sets and the natural numbers in the theory $\mathsf{ReR}^{\neg \infty}$.

\begin{theorem} \label{Th:InverseAckermannTotal}
($\mathsf{ReR}^{\neg \infty}$) For all $x$, there exists a unique natural number $z$ such that $z= \mathcal{I}_{\mathsf{Ack}}(x)$.
\end{theorem}

\begin{proof}
Work in the theory $\mathsf{ReR}^{\neg \infty}$. Note that the uniqueness clause of the theorem follows immediately from Lemma \ref{Th:UniquenessOfInverseAckermann}. To prove the existence part of the statement, use $\Pi_1\textsf{-Foundation}$ to find an $\in$-minimal element, $w$, of the class $\{ v \mid \neg \exists z (z = \mathcal{I}_{\mathsf{Ack}}(v)) \}$. Let $\theta(g, x)$ be the formula
\[
\mathsf{OP}(g) \land \mathsf{OP}(\mathsf{snd}(g)) \land \Psi_{\mathsf{IA}}(\mathsf{fst}(g), x, \mathsf{fst}(\mathsf{snd}(g)), \mathsf{snd}(\mathsf{snd}(g))). 
\]
Now, $(\forall u \in w)\exists! g \theta(g, u)$. Therefore, using $\Sigma_{\mathsf{u}}\textsf{-Replacement}$, let $c$ be such that for all $g$, $g \in c$ if and only if $(\exists u \in w)\theta(g, u)$. Using Theorem \ref{Th:TCoInReR}, let $y= \mathsf{tcl}(\{w\})$. Let $h^\prime$ be the function with domain $y$ such that for all $u \in y$, 
\[
h^\prime(u)= r \textrm{ if and only if } (\exists g \in c)(\mathsf{fst}(g)(u)= r).
\]
It is straightforward to check that $h^\prime$ exists as a set and is well-defined. Using Lemma \ref{Th:UniquenessOfPowerSum} and the fact that $w \subseteq y$, let
\[
z= \sum_{u \in w} 2^{h^\prime(u)}.
\]
Let $h= h^\prime \cup \{\langle w, z \rangle\}$. Then $\Psi_{\mathsf{IA}}(h, w, y, z)$, which contradicts our choice of $w$. \hfill$\square$
\end{proof}

\begin{corollary}
The formula $z= \mathcal{I}_{\mathsf{Ack}}(x)$ is $\Delta_{\mathsf{u}}^{\mathsf{ReR}^{\neg \infty}}$. \hfill$\square$
\end{corollary}

\begin{theorem} \label{Th:InverseAckermannSurjective}
($\mathsf{ReR}^{\neg \infty}$) For all natural numbers $z$, there exists a unique set $x$ such that $z= \mathcal{I}_{\mathsf{Ack}}(x)$. 
\end{theorem}

\begin{proof}
Work in the theory $\mathsf{ReR}^{\neg \infty}$. The uniqueness of $x$ follows from an induction argument utilising the uniqueness clause of Lemma \ref{Th:BinaryRepresentation}. 

To prove existence, suppose, for a contradiction, that the theorem is false. Using $\Pi_1\textsf{-Foundation}$, let $k$ be the least natural number such that $\neg \exists x (k= \mathcal{I}_{\mathsf{Ack}}(x))$. Using Lemma \ref{Th:BinaryRepresentation}, let $y$ be the set of natural numbers such that 
\[
k= \sum_{v \in y} 2^v.
\]
Our choice of $k$ ensures that $(\forall m \in y)\exists! u (m= \mathcal{I}_{\mathsf{Ack}}(u))$. Now, using $\Sigma_{\mathsf{u}}\textsf{-Replacement}$, let $x$ be a set such that for all $u$,
\[
u \in x \textrm{ if and only if } (\exists m \in y)(m= \mathcal{I}_{\mathsf{Ack}}(u)).
\]
This ensures that $k= \mathcal{I}_{\mathsf{Ack}}(x)$. \hfill$\square$   
\end{proof}

The bijection $\mathcal{I}_{\mathsf{Ack}}$ allows us to reduce an instance of $\Delta_0\textsf{-Collection}$ to an instance of $\Sigma_{\mathsf{u}}\textsf{-Replacement}$ showing that $\mathsf{ReR}^{\neg \infty}$ and $\mathsf{KP}^{\neg \infty}$ have the same consequences.

\begin{theorem} \label{Th:SameTheoryOfFiniteSets}
The theories $\mathsf{ReR}^{\neg \infty}$ and $\mathsf{KP}^{\neg \infty}$ have the same consequences.
\end{theorem}

\begin{proof}
We only need to verify that $\mathsf{ReR}^{\neg \infty}$ proves $\Delta_0\textsf{-Collection}$. Towards this end, 
let $\phi(x, y, \vec{z})$ be a $\Delta_0$-formula. Let $b$ and $\vec{a}$ be sets such that $(\forall x \in b) \exists y \phi(x, y, \vec{a})$. Let $\theta(x, f, \vec{z})$ be the formula
\[
\begin{array}{c}
(f \textrm{ is a function}) \land \mathsf{dom}(f)= n+1 \land\\
(\forall k \in n+1)(k= \mathcal{I}_{\mathsf{Ack}}(f(k))) \land \phi(x, f(n), \vec{z}) \land (\forall l \in n) \neg \phi(x, f(l), \vec{z}) 
\end{array}.
\]
The properties of $\mathcal{I}_{\mathsf{Ack}}$ ensure that $(\forall x \in b) \exists! f \theta(x, f, \vec{a})$. Using $\Delta_0\textsf{-Replacement}$, let $c$ be such that for all $f$,
\[
f \in c \textrm{ if and only if } (\exists x \in b) \theta(x, f, \vec{a}).
\]
Let $d= \bigcup^3 c$. Then $(\forall x \in b) (\exists y \in d) \phi(x, y, \vec{a})$. This shows that $\mathsf{ReR}^{\neg \infty}$ proves $\Delta_0\textsf{-Collection}$. \hfill$\square$   
\end{proof}

In particular, the proof of Theorem \ref{Th:SameTheoryOfFiniteSets} yields:

\begin{corollary} \label{Th:EverySetInBijectionWithSetOfNumbers}
($\mathsf{ReR}^{\neg \infty}$) For all $x$, there exists a set of natural numbers $y$ such that $|x|=|y|$. \hfill$\square$
\end{corollary}

The following result was proved by Vop\v{e}nca \cite{vop64} in the context of weak set theory including full \textsf{Replacement}. Our argument follows the proof of \cite[Proposition 2.13]{mat06}. 

\begin{lemma} \label{Th:PowersetsOfNumbersExist}
($\mathsf{ReR}$) For all natural numbers $n$, the powerset of $n$ exists.
\end{lemma}

\begin{proof}
Work in the theory $\mathsf{ReR}$. For all $x$ and $y$, define $F(x, y)=\{z \cup \{x\} \mid z \in y\}$. Note that $\Delta_0\textsf{-Replacement}$ ensure that for all $x$ and $y$, $F(x, y)$ exists. Now, consider the $\Delta_0$-formula $\phi(g, n)$ defined by
\[
\begin{array}{c}
(g \textrm{ is a function}) \land \mathsf{dom}(g)=n+1 \land\\
g(0)= \{\emptyset\} \land (\forall k \in n)(g(k+1)= g(k) \cup F(k, g(k)))
\end{array}.
\]
Now, $\Pi_1\textsf{-Foundation}$ ensures that for all natural numbers $n$, there exists $g$ such that $\phi(g, n)$ holds. Since for all natural numbers $n$, if $g$ is such that $\phi(g, n)$ holds, then $g(n)$ is the powerset of $n$, this proves the lemma. \hfill$\square$
\end{proof}

\begin{theorem} \label{Th:PowersetInReR}
($\mathsf{ReR}^{\neg \infty}$) \textsf{Powerset} \hfill$\square$
\end{theorem}

\begin{proof}
Work in the theory $\mathsf{ReR}^{\neg \infty}$. Let $x$ be a set. Using Corollary \ref{Th:EverySetInBijectionWithSetOfNumbers}, let $y$ be a set of natural numbers such that $|x|=|y|$ and let $f: x \longrightarrow y$ witness this bijection. Let $n$ be a natural number such that $y \subseteq n$. Note that $\mathcal{P}(y)= \{ z \in \mathcal{P}(n) \mid z \subseteq y\}$, which is a set by Lemma \ref{Th:PowersetsOfNumbersExist} and $\Delta_0\textsf{-Separation}$. Now, $\mathcal{P}(x)$ can be recovered from $\mathcal{P}(y)$ using the bijection $f$ and $\Delta_0\textsf{-Replacement}$. \hfill$\square$ 
\end{proof}

In particular, Theorem \ref{Th:PowersetInReR} shows that $\mathsf{ReR}^{\neg \infty}$ proves that for every $x$, the set of all finite subsets of $x$ is a set. Combined with \cite[Proposition 2.103]{mat06}, this provides a positive answer to \cite[Problem 8.26]{mat06}.

\begin{theorem}
($\mathsf{ReR}$) For all $x$, the set of finite subsets of $x$ is a set. \hfill$\square$.
\end{theorem}

Theorems \ref{Th:InverseAckermannTotal} and \ref{Th:InverseAckermannSurjective} also mean that the analysis in \cite[\S 2]{mz01} applies to the theory $\mathsf{ReR}$ showing that the set-theoretic analogue of Tennenbaum's Theorem applies to $\mathsf{ReR}$.

\begin{theorem}
Every recursive model of $\mathsf{ReR}$ is isomorphic to $V_\omega$. \hfill$\square$
\end{theorem}

\section[Model in which TCo fails]{A model of bounded collection in which $\mathsf{TCo}$ fails} \label{Sec:ModelInWhichTCoFails}

In this section we build a model of a significant fragment of \textsf{Collection} ($\Delta_0^{\mathcal{P}}\textsf{-Collection}$) in which the axiom of transitive containment fails. This shows that transitive containment is not provable in $\mathsf{ReR}_0+\textsf{Set-Foundation}$ or in the theory obtained by weakening $\Pi_1\textsf{-Foumdation}$ to $\textsf{Set-Foundation}$ in $\mathsf{KP}$. Our construction is based on the construction of a model of $\mathsf{Z}+\textsf{Foundation}$ presented in \cite[\S 12]{mat06}. To keep things simple, we will work in the theory $\mathsf{ZFC}+\mathsf{Con}(\mathsf{ZFC})$, but it is clear that a much weaker metatheory would suffice. We use $\iota$ to denote the operation $x \mapsto \{x\}$. Define $\iota^0 x= x$ and for all $n \in \omega$, $\iota^{n+1} x= \iota \iota^n x$.

Let $\mathcal{M}= \langle M, \in^{\mathcal{M}} \rangle$ be a model of $\mathsf{ZFC}$ with nonstandard natural numbers. Let $k_0 \in (\omega^{\mathcal{M}})^*$ be nonstandard. Let $c \in M$ be such that
\[
\mathcal{M} \models (c= \{\iota^m \emptyset \mid m \geq k_0\}).
\]
Working within $\mathcal{M}$, define $K_0= \omega \cup c$ and for all $n \in \omega$,
\[
K_{n+1}= K_n\cup \mathcal{P}(K_n) \cup \left(\bigcup K_n \right).
\]
Working in the metatheory, define $\mathcal{K}= \langle K, \in^{\mathcal{K}}\rangle$ by
\[
K= \bigcup_{n \in \omega} (K_n^{\mathcal{M}})^* \textrm{ and } \in^{\mathcal{K}} \textrm{ is the restriction of } \in^{\mathcal{M}} \textrm{ to } K.
\]
Note that the union in the above definition is taken over the standard natural numbers.

\begin{lemma} \label{Th:KSuperTransitiveSubstructure}
$\mathcal{K} \subseteq_{\mathsf{tr}}^{\mathcal{P}} \mathcal{M}$.
\end{lemma} 

\begin{proof}
To see that $\mathcal{K} \subseteq_{\mathsf{tr}} \mathcal{M}$, let $x \in K$ and let $y \in M$ be such that $\mathcal{M} \models (y \in x)$. Let $n \in \omega$ be such that $\mathcal{M} \models (x \in K_n)$. So,
\[
\mathcal{M} \models \left( y \in \bigcup K_n \subseteq K_{n+1} \right),
\]
which shows that $y \in K$.

Now, let $x \in K$ and let $y \in M$ be such that $\mathcal{M} \models (y \subseteq x)$. Let $n \in \omega$ be such that $\mathcal{M} \models (x \in K_n)$. Now,
\[
\mathcal{M} \models \left( y \subseteq x \subseteq \bigcup K_n \subseteq K_{n+1}\right).
\] 
So, $\mathcal{M} \models (y \in \mathcal{P}(K_{n+1}) \subseteq K_{n+2})$ and $y \in K$. \hfill$\square$ 
\end{proof} 

\begin{lemma} \label{Th:PowersetsInK}
For all $x \in K$, $\mathcal{P}^{\mathcal{M}}(x) \in K$.
\end{lemma}

\begin{proof}
Let $x \in K$. Let $n \in \omega$ be such that $\mathcal{M} \models (x \in K_n)$. The proof of Lemma \ref{Th:KSuperTransitiveSubstructure} shows that $\mathcal{M} \models (\mathcal{P}(x) \subseteq K_{n+2})$. Therefore, $\mathcal{M} \models (\mathcal{P}(x) \in K_{n+3})$ and $\mathcal{P}^{\mathcal{M}}(x) \in K$. \hfill$\square$ 
\end{proof}

\begin{lemma} \label{Th:UnionsInK}
For all $x \in K$, $\bigcup x \in K$.
\end{lemma}

\begin{proof}
Let $x \in K$. Let $n \in \omega$ be such that $\mathcal{M} \models (x \in K_n)$. Now,
\[
\mathcal{M} \models \left( x \subseteq \bigcup K_n \subseteq K_{n+1} \right) \textrm{ and } \mathcal{M} \models \left( \bigcup x \subseteq \bigcup K_{n+1} \subseteq K_{n+2}\right).
\] 
Therefore, $\mathcal{M} \models \left( \bigcup x \in K_{n+3} \right)$ and $\bigcup x \in K$. \hfill$\square$
\end{proof}

\begin{lemma} \label{Th:PairingInK}
For all $x, y \in K$, $\{x, y\} \in K$.
\end{lemma}

\begin{proof}
Let $x, y \in K$. Let $n \in \omega$ be such that $\mathcal{M} \models (x, y \in K_n)$. Therefore $\mathcal{M} \models \{x, y \} \in K_{n+1})$ and $\{x, y\} \in K$. \hfill$\square$
\end{proof}

\begin{lemma} 
For all $n \in \omega$, $K_n^{\mathcal{M}} \in K$.
\end{lemma}

\begin{proof}
Let $n \in \omega$. We have $\mathcal{M} \models (K_n \subseteq K_n)$, so $\mathcal{M} \models (K_n \in K_{n+1})$ and $K_n^{\mathcal{M}} \in K$. \hfill$\square$
\end{proof}

\begin{theorem}
The structure $\mathcal{K}$ satisfies $\mathsf{S}_0+\textsf{Infinity}+\textsf{Powerset}+\Delta_0^{\mathcal{P}}\textsf{-Separation}+\Delta_0^{\mathcal{P}}\textsf{-Collection}+\textsf{Set-Foundation}+\mathsf{AC}$.
\end{theorem}

\begin{proof}
Lemmas \ref{Th:KSuperTransitiveSubstructure}, \ref{Th:PowersetsInK}, \ref{Th:UnionsInK} and \ref{Th:PairingInK} immediately imply that $\mathcal{K}$ satisfies $\mathsf{S}_0+\textsf{Powerset}+\Delta_0^{\mathcal{P}}\textsf{-Separation}+\textsf{Set-Foundation}+\mathsf{AC}$. Since $\mathcal{M} \models (\omega \subseteq K_0)$, $\mathcal{M} \models (\omega \in K_1)$ and $\omega^{\mathcal{M}} \in K$. Therefore, $\textsf{Infinity}$ holds in $\mathcal{K}$. We are left to verify that $\Delta_0^{\mathcal{P}}\textsf{-Collection}$ holds in $\mathcal{K}$. Towards this end, let $\phi(x, y, \vec{z})$ be a $\Delta_0^{\mathcal{P}}$-formula. Let $b, \vec{a} \in K$ be such that 
\begin{equation} \label{eq:CollectionAntecedent}
\mathcal{K} \models (\forall x \in b)\exists y \phi(x, y, \vec{a}).
\end{equation}
Work inside $\mathcal{M}$. Let
\[
q= \{ \langle x, n \rangle \in b \times \omega \mid (\exists y \in K_n) \phi(x, y, \vec{a}) \land (\forall m \in n)(\forall z \in K_m) \neg \phi(x, z, \vec{a})\},
\]
and let $d= \mathsf{rng}(q)$.

Work in the metatheory again. Now, by (\ref{eq:CollectionAntecedent}), $d^* \subseteq \omega$. Therefore, there exists $l \in \omega$ such that for all $n \in d^*$, $n < l$, otherwise the standard natural numbers would be a set in $\mathcal{M}$. So, 
\[
\mathcal{K} \models (\forall x \in b)(\exists y \in K_l)\phi(x, y, \vec{a}).
\]
This shows that $\Delta_0^{\mathcal{P}}\textsf{-Collection}$ holds in $\mathcal{K}$. \hfill$\square$  
\end{proof}

Finally, we need to show that $\mathsf{TCo}$ does not hold in $\mathcal{K}$. 

\begin{lemma} \label{Th:InteratedSingletonsInK}
For all $n \in \omega$,
\[
\mathcal{M} \models (\forall m \in \omega)(\iota^m \emptyset \in K_n \iff m \leq n+1 \lor k_0-n \leq m).
\]
\end{lemma}

\begin{proof}
We prove this by induction on $n$. In $\mathcal{M}$, for all $m \in \omega$, $\iota^m \emptyset \in K_0$ if and only if $m=0$, $m=1$ or $k_0 \leq m$. This is the base case of the induction. Now, suppose that $n \in \omega$ is such that
\[ 
\mathcal{M} \models (\forall m \in \omega)(\iota^m \emptyset \in K_n \iff m \leq n+1 \lor k_0-n \leq m).
\]
Work inside $\mathcal{M}$. Now,
\[
K_{n+1}= K_n \cup \mathcal{P}(K_n) \cup \left(\bigcup K_n \right).
\]
Therefore, for all $m \in \omega$, $\iota^m \emptyset \in K_{n+1}$ if and only if $m \leq n+2$ or $k_0-(n+1) \leq m$. The lemma now follow by induction. \hfill$\square$ 
\end{proof} 

\begin{theorem}
$\mathcal{K} \models (\iota^{k_0} \emptyset \textrm{ is not conatined in any transitive set})$.
\end{theorem}

\begin{proof}
Suppose, for a contradiction, that $T \in K$ is such that
\[
\mathcal{K} \models (\iota^{k_0} \emptyset \subseteq T) \land (T \textrm{ is transitive}).
\]
Using the fact that $\omega^{\mathcal{M}}$ is nonstandard, let $k_1 \in (\omega^{\mathcal{M}})^*$ be such that for all $n \in \omega$, $\mathcal{M} \models (k_1 < k_0-n)$. Therefore, by Lemma \ref{Th:InteratedSingletonsInK}, $\iota^{k_1} \emptyset \notin K$.   

Work inside $\mathcal{M}$. Note that, by Lemma \ref{Th:KSuperTransitiveSubstructure}, $T$ is transitive and $\iota^{k_0} \emptyset \subseteq T$. Consider 
\[
A= \{ m \in \omega \mid (k_1 \leq m < k_0) \land (\iota^m \emptyset \in T)\}.
\]
Since $\iota^{k_0} \emptyset \subseteq T$ and $\iota^{k_0-1} \emptyset \in \iota^{k_0} \emptyset$, $k_0-1 \in A$. Let $q$ be the least element of $A$. Since $iota^{k_1} \emptyset \notin K$, $\iota^{k_1} \emptyset \notin T$ and $q > k_1$. But $\iota^q \emptyset \in T$ and, since $q-1 \geq k_1$, $\iota^{q-1} \emptyset \notin T$, which contradicts the fact that $T$ is transitive. \hfill$\square$ 
\end{proof}

\section[Questions]{Questions}

In \cite{mat06}, Mathias introduces two weakenings of the scheme of $\Delta_0\textsf{-Collection}$:
\begin{itemize}
\item[]($\textsf{Flat } \Delta_0\textsf{-Replacement}$): For all $\Delta_0$-formulae, $\phi(x, y)$,
\[
\forall z \forall u ((\forall x \in u) \exists ! y (\phi(x, y) \land y \subseteq z) \Rightarrow \exists v \forall y( y \in v \iff (\exists x \in u) (\phi(x, y) \land y \subseteq z))).
\]
\item[]($\textsf{Flat } \Delta_0\textsf{-Collection}$):
\[
\forall z \forall u ((\forall x \in u)\exists y (\phi(x, y) \land y \subseteq z) \Rightarrow \exists v (\forall x \in u) (\exists y \in v) (\phi(x, y) \land y \subseteq z)).
\] 
\end{itemize}
\begin{itemize}
\item The theory $\mathsf{fReR}$ is obtained from $\mathsf{ReR}$ by replacing $\Delta_0\textsf{-Replacement}$ with $\textsf{Flat }\Delta_0\textsf{-Replacement}$.
\item The theory $\mathsf{fReC}$ is obtained from $\mathsf{ReR}$ by replacing $\Delta_0\textsf{-Replacement}$ with $\textsf{Flat }\Delta_0\textsf{-Collection}$ and $\Delta_0\textsf{-Separation}$.
\item The theories $\mathsf{fReR}^{\neg \infty}$ and $\mathsf{fReC}^{\neg \infty}$ are obtained from $\mathsf{fReR}$ and $\mathsf{fReC}$, respectively, by adding $\neg \mathsf{Infinity}$.  
\end{itemize}
Note that $\mathsf{fReR}$ is subsystem of $\mathsf{fReC}$. Since the scheme $\textsf{Flat } \Delta_0\textsf{-Collection}$ is provable in $\mathsf{Z}$, the model presented in \cite[\S 12]{mat06} shows that $\mathsf{fReC}$ does not prove $\mathsf{TCo}$.

\begin{question}
Do the theories $\mathsf{fReR}^{\neg \infty}$ and $\mathsf{fReC}^{\neg \infty}$ prove $\mathsf{TCo}$?
\end{question}
  
Theorem \ref{Th:EverySetInBijectionWithSetOfNumbers} shows that the theory $\mathsf{ReR}^{\neg \infty}$ proves that every set is in bijection with a natural number. In \cite[Corollary 10]{bf93} and \cite[Remark 2.2(d)]{esv09}, it is observed that $\mathsf{ReR}_0+\textsf{Replacement}+\textsf{Powerset}+\neg \textsf{Infinity}$, which has access to the powersets but no foundation, proves that every set is in bijection with a natural number. In contrast, Kunen \cite[\S 7]{bf93}, shows that $\mathsf{ReR}_0+\textsf{Replacement}+\neg\textsf{Infinity}$ does not prove that every set is in bijection with a natural number.

\begin{question}
Do the theories $\mathsf{fReR}^{\neg \infty}$ and $\mathsf{fReC}^{\neg \infty}$ prove the totality of the inverse Ackermann interpretation?
\end{question}

\begin{question}
Do the theories $\mathsf{fReR}^{\neg \infty}$ and $\mathsf{fReC}^{\neg \infty}$ prove that every set is in bijection with a natural number?
\end{question}

Let $\mathsf{KP}_-$ be the theory obtained from $\mathsf{KP}$ by replacing $\Pi_1\textsf{-Foundation}$ with $\textsf{Set-Foundation}$. Zambella \cite{zam98} shows that, despite not proving $\mathsf{TCo}$, $\mathsf{KP}_-$ prove $\Gamma\textsf{-Foundation}$, where $\Gamma$ is the collection of quantifier-free formulae. The following question from \cite{zam98} remains open:

\begin{question}
Does $\mathsf{KP}_-$ prove $\Delta_0\textsf{-Foundation}$? 
\end{question}

\end{document}